\documentclass[12pt]{amsart}

\usepackage{amssymb}
\usepackage{enumerate}
\usepackage{geometry}
\usepackage{graphicx}

\makeatletter
\@namedef{subjclassname@2010}{%
  \textup{2010} Mathematics Subject Classification}
\makeatother

\newtheorem{thm}{Theorem}[section]
\newtheorem{cor}[thm]{Corollary}
\newtheorem{lmm}[thm]{Lemma}
\newtheorem{prp}[thm]{Proposition}

\theoremstyle{definition}
\newtheorem{dfn}[thm]{Definition}

\numberwithin{equation}{section}

\begin{document}

\baselineskip=17pt


\title{Real Zeros of the Hurwitz zeta function}

\author{Toshiki Matsusaka}
\address{Graduate School of Mathematics, Kyushu University, Motooka 744, Nishi-ku Fukuoka 819-0395, Japan}
\email{toshikimatsusaka@gmail.com}

\date{}

\begin{abstract}
It is well known that real zeros of the Riemann zeta function are negative even integers. As for real zeros of the Hurwitz zeta function, T. Nakamura recently gave an existence condition in the intervals $(0,1)$ and $(-1,0)$. We generalize this result for all negative real numbers.
\end{abstract}



\maketitle

\section{Introduction}\label{s1} 
The Hurwitz zeta function was first introduced by Hurwitz \cite{Hur} as a generalization of the Riemann zeta function $\zeta(s) = \sum_{n=1}^{\infty} n^{-s}$.

\begin{dfn}
Let $0<a\leq1$, $s \in \mathbb{C}$. Then the Hurwitz zeta function $\zeta(s, a)$ is defined by
\begin{align*}
\zeta(s,a) := \sum_{n=0}^{\infty}(n+a)^{-s},\quad s:=\sigma + it,\quad \sigma >1,\quad t\in \mathbb{R}.
\end{align*}
\end{dfn}

This series converges absolutely in the half-plane $\sigma > 1$ and uniformly in each compact subset of this half-plane. Moreover $\zeta(s,a)$ is analytically continued to the whole complex plane except for a simple pole at $s=1$. We easily see that the Hurwitz zeta function has no real zeros in this half-plane $\sigma > 1$. For the special case of $a = 1$, the Riemann zeta function $\zeta(s) = \zeta(s,1)$ has real zeros at the negative even integers called the trivial zeros. In general, Nakamura \cite{Naka1, Naka2} investigated real zeros of $\zeta(s,a)$ in the intervals $(0,1)$ and $(-1,0)$, and obtained the following existence conditions.

\begin{thm}\label{Naka}
$($\cite{Naka1, Naka2}$)$ Let $b_2^{\pm}:=(3\pm\sqrt{3})/6$. Then
\begin{itemize}
\item[$(1)$]$\zeta(\sigma,a)$ has real zeros in the interval $(0,1)$ if and only if $0 < a < 1/2$.
\item[$(2)$]$\zeta(\sigma,a)$ has real zeros in the interval $(-1,0)$ if and only if $0< a < b_2^{-}$ or $1/2 < a < b_2^{+}$.
\end{itemize}
Note that $b_2^{\pm}$ are the roots of the second Bernoulli polynomial $B_2(x):=x^2-x+1/6$.
\end{thm}

In this article, we generalize this result to general intervals in the negative real numbers.

\begin{thm}\label{main1}
Let $N \geq -1$ be an integer. Then $\zeta(\sigma,a)$ has real zeros in the interval $(-N-1,-N)$ if and only if $B_{N+1}(a)B_{N+2}(a) < 0$, where $B_n(x)$ is the $n$th Bernoulli polynomial $($defined in Section $\ref{s2})$.
\end{thm}

We can easily see that Theorem \ref{Naka} is a special case of Theorem \ref{main1}. Furthermore, we can state this theorem explicitly as follows.

\begin{thm}\label{main2}
Let $N \geq 0$ be an integer. Then $\zeta(\sigma,a)$ has real zeros in the interval $(-N-1,-N)$ if and only if
\begin{align*}
\left\{\begin{array}{ll}
0 < a < b_{N+2}^{-}$ or $1/2 < a < b_{N+2}^{+}\quad &$if N is even$, \\
b_{N+1}^{-} < a < 1/2$ or $b_{N+1}^{+} < a < 1\quad &$if N is odd$,
\end{array} \right.
\end{align*}
where $b_n^{\pm}$ are the two roots of the $n$th Bernoulli polynomial $B_n(x)$ in the interval $(0,1)$.
\end{thm}

In the particular cases of $N \geq 4$, we can show the uniqueness of the zero in each interval. More recently, Endo and Suzuki \cite{ES} revealed the uniqueness of the zero in (0,1) and its asymptotic behavior with respect to $a$. In Section 2, we show some properties of the Bernoulli polynomials and the Hurwitz zeta function. In Section 3, we prove Theorem \ref{main1} for the cases of $1 \leq N \leq 3$ and $N \geq 4$, separately. Note that the known results on zeros of the Hurwitz zeta function are reviewed in Nakamura \cite[Section 1.2]{Naka2}. 

\section{Preliminaries}\label{s2} 
In this section, we recall some properties of the Bernoulli polynomials and the Hurwitz zeta function. First, we define the Bernoulli numbers and the Bernoulli polynomials.

\begin{dfn}$($see \cite[Chapter 1, 4]{AIK}$)$
The Bernoulli numbers $B_n$ and the Bernoulli polynomials $B_n(x)$ are defined by means of generating functions,
\begin{align*}
\sum_{n=0}^{\infty}B_n\frac{t^n}{n!} &= \frac{te^t}{e^t-1},\\
\sum_{n=0}^{\infty}B_n(x)\frac{t^n}{n!} &= \frac{te^{xt}}{e^t-1}.
\end{align*}
In particular, $B_n(1) = B_n$ holds.
\end{dfn}

It is known that the $n$th Bernoulli polynomial $B_n(x)$ has exactly one root in each interval $[0,1/2)$ and $[1/2,1)$ for $n \geq 2$, (see \cite[Section 1]{Dav}). Then, we can define $b_n^{-}$ and $b_n^{+}$ as the roots of $B_n(x)$ in the intervals $[0,1/2)$ and $[1/2,1)$, respectively. Especially, we have $b_n^- = 0$ and $b_n^+ = 1/2$ for odd integers $n \geq 3$. Moreover, the following holds.

\begin{lmm}\label{Bpm}
Let $k$ be a positive integer. Then we have\\
$(1)$ $(-1)^{k-1}B_{2k}(x) > 0\ (0 \leq x < b_{2k}^-, b_{2k}^+ < x \leq 1)$, $(-1)^{k-1}B_{2k}(x) < 0\ (b_{2k}^- < x < b_{2k}^+)$.\\
$(2)$ $(-1)^{k-1}B_{2k+1}(x) > 0\ (0<x<1/2)$, $(-1)^{k-1}B_{2k+1}(x) < 0\ (1/2 < x <1)$.
\end{lmm}

\begin{proof}
By Proposition 4.9 in \cite{AIK}, we have $B_n(0) = B_n(1) = B_n\ (n \neq 1)$, and $B_n'(x) = nB_{n-1}(x)\ (n \geq 1)$. In particular, for each positive integer $k$, we have $B_{2k+1}(0) = B_{2k+1}(1/2) = B_{2k+1}(1) =0$. On the other hand, by Corollary 1.16 in \cite{AIK}, it holds that $(-1)^{k-1}B_{2k} > 0$ for each positive integer $k$. Combining these results, we can obtain this lemma.
\end{proof}

By using Lemma \ref{Bpm}, the equivalence of Theorem \ref{main1} and Theorem \ref{main2} for $N\geq 0$ follows directly. For further properties of the Bernoulli polynomials, see \cite{AIK, Dav}. Next, we show an integral representation of the Hurwitz zeta function.

\begin{prp}\label{IR}
Let $N \geq -1$ be an integer. For $-N-1 < \mathrm{Re}(s) < -N$, it holds that
\begin{align*}
\Gamma(s)\zeta(s,a) = \int_0^{\infty}\biggl( \frac{e^{(1-a)x}}{e^x-1} - \sum_{n=0}^{N+1}\frac{B_n(1-a)}{n!}x^{n-1}\biggr) x^{s-1} dx.
\end{align*}
\end{prp}

\begin{proof}
It is known that the Hurwitz zeta function has the integral representation
\begin{align*}
\zeta(s,a) = \frac{1}{\Gamma(s)}\int_0^{\infty}\frac{e^{(1-a)x}}{e^x-1}x^{s-1}dx,\quad \mathrm{Re}(s) > 1,
\end{align*}
(see \cite[Section 9]{AIK}). Note that the integrand function can be expanded as the Laurent series
\begin{align*}
\frac{e^{(1-a)x}}{e^x-1} = \sum_{n=0}^{\infty} \frac{B_n(1-a)}{n!}x^{n-1}
\end{align*}
for $0 < |x| < 2\pi$.

First, we meromorphically continue the function $\Gamma(s)\zeta(s,a)$ to the half-plane $\mathrm{Re}(s) > -N-1$ by using this integral representation. We divide this integral into three parts,
\begin{align*}
\Gamma(s)\zeta(s,a) &= P(s) + Q_N(s) + R_N(s),\\
P(s) &= \int_1^{\infty} \frac{e^{(1-a)x}}{e^x-1} x^{s-1} dx,\\
Q_N(s) &= \int_0^1 \biggl( \sum_{n=0}^{N+1} \frac{B_n(1-a)}{n!}x^{n-1} \biggr) x^{s-1} dx,\\
R_N(s) &= \int_0^1 \biggl( \frac{e^{(1-a)x}}{e^x-1}-\sum_{n=0}^{N+1} \frac{B_n(1-a)}{n!}x^{n-1} \biggr) x^{s-1} dx.
\end{align*}
Then $P(s)$ is holomorphic in the whole complex plane, $Q_N(s)$ is a rational function, and $R_N(s)$ is holomorphic in the half-plane $\mathrm{Re}(s) > -N-1$, that is, $\Gamma(s)\zeta(s,a)$ is meromorphically continued to the half-plane $\mathrm{Re}(s) > -N-1$.\\

Next we restrict the domain to the strip $-N-1 < \mathrm{Re}(s) < -N$. Then we can write
\begin{align*}
Q_N(s) &= \sum_{n=0}^{N+1} \frac{B_n(1-a)}{n!} \frac{1}{n+s-1}\\
&= -\sum_{n=0}^{N+1} \frac{B_n(1-a)}{n!} \int_1^{\infty}x^{n+s-2} dx\\
&= -\int_1^{\infty} \sum_{n=0}^{N+1} \frac{B_n(1-a)}{n!} x^{n+s-2} dx,
\end{align*}
the last integral being absolutely convergent. We therefore obtain
\begin{align*}
P(s)+Q_N(s)+R_N(s) = \int_0^{\infty} \biggl( \frac{e^{(1-a)x}}{e^x-1}-\sum_{n=0}^{N+1} \frac{B_n(1-a)}{n!}x^{n-1} \biggr) x^{s-1} dx
\end{align*}
for $-N-1 < \mathrm{Re}(s) < -N$.
\end{proof}

Finally, we remark the classical relation between the Bernoulli polynomial and the Hurwitz zeta function.

\begin{lmm}$($\cite[Proposition 9.3]{AIK}$)$\label{BH}
For each integer $n>0$, it holds that
\begin{align*}
\zeta(1-n,a) = -\frac{B_n(a)}{n}.
\end{align*}
\end{lmm}

\section{Proof of Theorem \ref{main1}}\label{s3} 
We give a proof for the cases of $1 \leq N \leq 3$ and $N \geq 4$.\\

\noindent (1) $1 \leq N \leq 3$ case.\\

For each integer $N \geq -1$, We put
\begin{align*}
G_N(a,x) := \frac{e^{(1-a)x}}{e^x-1} - \sum_{n=0}^{N+1}\frac{B_n(1-a)}{n!}x^{n-1}.
\end{align*}
Then we have $\Gamma(s)\zeta(s,a) = \int_0^{\infty}G_N(a,x)x^{s-1}dx$ for $-N-1 < \mathrm{Re}(s) < -N$ by Proposition \ref{IR}. Moreover we put
\begin{align*}
g_N(a,x):=x(e^x-1)G_N(a,x) = xe^{(1-a)x} - \sum_{n=0}^{N+1}\frac{B_n(1-a)}{n!}x^n (e^x-1),
\end{align*}
and differentiate it $N+2$ times. By the definition of the Bernoulli polynomials, $g_N(a,x)$ has a zero of order $N+3$ at $x=0$. Then we get $g_N^{(k)}(a,0)$ = 0 for $k$ = 0, 1, 2, $\dots$, $N+2$. By direct calculation, we have
\begin{align*}
g_N^{(N+2)}(a,x) &= (N+2)(1-a)^{N+1}e^{(1-a)x}+(1-a)^{N+2}xe^{(1-a)x}\\
&\hspace{60pt} -\sum_{n=0}^{N+1}\frac{B_n(1-a)}{n!}\sum_{k=0}^n \binom{N+2}{k}\binom{n}{k}k!x^{n-k} e^x.
\end{align*}
Putting $m = n-k$, we can get the form
\begin{align*}
e^{(a-1)x}g_N^{(N+2)}(a,x) &= (N+2)(1-a)^{N+1}+(1-a)^{N+2}x\\
&\hspace{45pt} -\sum_{m=0}^{N+1}\Biggl( \sum_{k=0}^{N+1-m}\binom{N+2}{k} B_{m+k}(1-a) \Biggr) \frac{x^m}{m!} \cdot e^{ax}.
\end{align*}

\noindent (1-i) For the case of $N$ = 1, we have
\begin{align*}
&e^{(a-1)x}g_1^{(3)}(a,x)\\
&\hspace{5pt} = 3(1-a)^2+(1-a)^3x-\Biggl( 3(1-a)^2+(1-3a)(1-a)x+\frac{1}{2}B_2(1-a)x^2\Biggr) \cdot e^{ax}\\
&\hspace{5pt} = -2a(1-a)(1-2a)x\\
&\hspace{40pt} + \sum_{n=2}^{\infty}a^{n-2}\Biggl( -3(1-a)^2a^2-n(1-3a)(1-a)a-\frac{n(n-1)}{2}B_2(1-a) \Biggr) \frac{x^n}{n!}.\\
\end{align*}

Let $0 < a \leq b_2^{-}$. Then we have $-2a(1-a)(1-2a) < 0$. For $n \geq 2$, we have
\begin{align*}
&-3(1-a)^2a^2 < 0,\\
&-n(1-3a)(1-a)a < 0,\\
&-\frac{n(n-1)}{2}B_2(1-a) \leq 0.
\end{align*}
By these inequalities, it holds that $e^{(a-1)x}g_1^{(3)}(a,x) <0$ for all $x >0$. By $g_1(a,0) = g_1^{(1)}(a,0) = g_1^{(2)}(a,0) = g_1^{(3)}(a,0) = 0$, we obtain $g_1(a,x) <0$, $G_1(a,x) <0$. Then we have $\Gamma(\sigma)\zeta(\sigma, a) < 0$ for all $\sigma \in (-2,-1)$. In general, for an integer $k \geq 0$, it holds that $(-1)^{k-1}\Gamma(\sigma) > 0$ for all $\sigma \in (-k-1,-k)$. Thus we get $\zeta(\sigma,a) <0$ for all $\sigma \in (-2,-1)$, that is, $\zeta(\sigma,a)$ has no real zeros in the interval $(-2,-1)$. \\

Let $1/2 \leq a \leq b_2^{+}$. Then we have $-2a(1-a)(1-2a) \geq 0$. For $n \geq 2$, we have
\begin{align*}
&-n(1-3a)(1-a)a > 0,\\
&-\frac{n(n-1)}{2}B_2(1-a) \geq 0.
\end{align*}
For the remaining part, it holds that
\begin{align*}
-3(1-a)^2a^2-n(1-3a)(1-a)a &\geq -3(1-a)^2a^2-2(1-3a)(1-a)a\\
&= -a(1-a)(2-3a-3a^2) >0.
\end{align*}
Thus we obtain $e^{(a-1)x}g_1^{(3)}(a,x) >0$, that is, $\zeta(\sigma,a) >0$ for all $\sigma \in (-2,-1)$ in the same way.\\

Finally let $b_2^{-} < a < 1/2$ or $b_2^{+} < a < 1$, then it holds that $B_2(a)B_3(a) < 0$. By Lemma \ref{BH}, we see that $\zeta(-1,a) = -B_2(a)/2$ and $\zeta(-2,a) = -B_3(a)/3$. By assumption, we have $\zeta(-1,a)\zeta(-2,a)<0$, that is, $\zeta(\sigma,a)$ has zeros in $(-2,-1)$. In the case of $a=1$, the Riemann zeta function $\zeta(\sigma) = \zeta(\sigma,1)$ has no real zeros in $(-2,-1)$. This concludes the proof for the case of $N$ =1.\\

\noindent (1-ii) For the case of $N$ = 2, we have
\begin{align*}
&e^{(a-1)x}g_2^{(4)}(a,x)\\
&\hspace{5pt} = -\frac{1}{6}(1-30a^2+60a^3-30a^4)x-\frac{1}{12}(1-4a-30a^2+60a^3-24a^5)x^2\\
&\hspace{25pt}+ \sum_{n=3}^{\infty}a^{n-3}\Biggl( -4(1-a)^3a^3-\frac{n}{6}(7-6a)(1-6a+6a^2)a^2\\
&\hspace{35pt}-\frac{n(n-1)}{12}(1-18a+42a^2-24a^3)a-\frac{n(n-1)(n-2)}{6}B_3(1-a) \Biggr) \frac{x^n}{n!}.
\end{align*}

Let $b_4^{-} \leq a \leq 1/2$. Then we have
\begin{align*}
&-\frac{1}{6}(1-30a^2+60a^3-30a^4) \geq 0,\\
&-\frac{1}{12}(1-4a-30a^2+60a^3-24a^5) > 0.
\end{align*}
For $n \geq 3$, we have
\begin{align*}
&-\frac{n}{6}(7-6a)(1-6a+6a^2)a^2 > 0,\\
&-\frac{n(n-1)}{12}(1-18a+42a^2-24a^3)a > 0,\\
&-\frac{n(n-1)(n-2)}{6}B_3(1-a) \geq 0.
\end{align*}
For the remaining part, it holds that
\begin{align*}
&-4(1-a)^3a^3-\frac{n(n-1)}{12}(1-18a+42a^2-24a^3)a\\
&\hspace{15pt} \geq -4(1-a)^3a^3-\frac{3(3-1)}{12}(1-18a+42a^2-24a^3)a\\
&\hspace{15pt} = -\frac{1}{2}(1-18a+50a^2-48a^3+24a^4-8a^5)a > 0.
\end{align*}
Thus we obtain $e^{(a-1)x}g_2^{(4)}(a,x) >0$ for all $x >0$, that is, $\zeta(\sigma,a)<0$ for all $\sigma \in (-3,-2)$.\\

Let $b_4^{+} \leq a \leq 1$. Then we have
\begin{align*}
&-\frac{1}{6}(1-30a^2+60a^3-30a^4) \leq 0,\\
&-\frac{1}{12}(1-4a-30a^2+60a^3-24a^5) < 0.
\end{align*}
For $n \geq 3$, we have
\begin{align*}
&-4(1-a)^3a^3 \leq 0,\\
&-\frac{n(n-1)}{12}(1-18a+42a^2-24a^3)a < 0,\\
&-\frac{n(n-1)(n-2)}{6}B_3(1-a) \leq 0.
\end{align*}
For the remaining part, it holds that
\begin{align*}
&-\frac{n}{6}(7-6a)(1-6a+6a^2)a^2-\frac{n(n-1)}{12}(1-18a+42a^2-24a^3)a\\
&\hspace{15pt} \leq n\Biggl\{-\frac{1}{6}(7-6a)(1-6a+6a^2)a^2-\frac{(3-1)}{12}(1-18a+42a^2-24a^3)a\Biggr\}\\
&\hspace{15pt} =-\frac{n}{6}(1-11a-6a^2+54a^3-36a^4)a < 0.
\end{align*}
Thus we obtain $e^{(a-1)x}g_2^{(4)}(a,x) <0$, that is, $\zeta(\sigma,a) >0$ for all $\sigma \in (-3,-2)$.\\

Finally let $0<a<b_4^{-}$ or $1/2<a < b_4^{+}$, $\zeta(\sigma,a)$ has zeros in $(-3,-2)$ in the same way as (1-i). This concludes the proof for the case of $N$ =2. \\

\noindent (1-iii) For the case of $N$ = 3, we have
\begin{align*}
&e^{(a-1)x}g_3^{(5)}(a,x)\\
&\hspace{5pt}= a(1-a)(1-2a)(1+3a-3a^2)x\\
&\hspace{30pt} +\frac{1}{12}(1+9a-9a^2-90a^3+120a^4-30a^6)x^2\\
&\hspace{35pt} +\frac{1}{36}(1+6a+6a^2-99a^3+180a^5-60a^6-30a^7)x^3\\
&\hspace{40pt} + \sum_{n=4}^{\infty}a^{n-4}\Biggl( -5(1-a)^4a^4-n(1-a)(1-10a+20a^2-10a^3)a^3\\
&\hspace{70pt} +\frac{n(n-1)}{12}(1+21a-111a^2+150a^3-60a^4)a^2\\
&\hspace{75pt} +\frac{n(n-1)(n-2)}{36}(1+3a-39a^2+66a^3-30a^4)a\\
&\hspace{80pt} -\frac{n(n-1)(n-2)(n-3)}{24}B_4(1-a) \Biggr) \frac{x^n}{n!}.
\end{align*}

Let $0 < a \leq b_4^{-}$. Then we have
\begin{align*}
&a(1-a)(1-2a)(1+3a-3a^2) > 0,\\
&\frac{1}{12}(1+9a-9a^2-90a^3+120a^4-30a^6) > 0,\\
&\frac{1}{36}(1+6a+6a^2-99a^3+180a^5-60a^6-30a^7) > 0.
\end{align*}
For $n \geq 4$, we have
\begin{align*}
&\frac{n(n-1)}{12}(1+21a-111a^2+150a^3-60a^4)a^2 > 0,\\
&\frac{n(n-1)(n-2)}{36}(1+3a-39a^2+66a^3-30a^4)a > 0,\\
&-\frac{n(n-1)(n-2)(n-3)}{24}B_4(1-a) \geq 0.
\end{align*}
For the remaining parts, we have
\begin{align*}
&-5(1-a)^4a^4+\frac{n(n-1)}{12}(1+21a-111a^2+150a^3-60a^4)a^2\\
&\hspace{15pt} \geq -5(1-a)^4a^4+\frac{4(4-1)}{12}(1+21a-111a^2+150a^3-60a^4)a^2\\
&\hspace{15pt} =(1+21a-116a^2+170a^3-90a^4+20a^5-5a^6)a^2 > 0,
\end{align*}
and
\begin{align*}
&-n(1-a)(1-10a+20a^2-10a^3)a^3\\
&\hspace{60pt} +\frac{n(n-1)(n-2)}{36}(1+3a-39a^2+66a^3-30a^4)a\\
&\hspace{15pt} \geq n\Biggl\{-(1-a)(1-10a+20a^2-10a^3)a^3\\
&\hspace{60pt} +\frac{(4-1)(4-2)}{36}(1+3a-39a^2+66a^3-30a^4)a\Biggr\}\\
&\hspace{15pt} =\frac{n}{6}(1+3a-45a^2+132a^3-210a^4+180a^5-60a^6)a >0.
\end{align*}
Thus we obtain $e^{(a-1)x}g_3^{(5)}(a,x) >0$ for all $x >0$, that is, $\zeta(\sigma,a) >0$ for all $\sigma \in (-4,-3)$. \\

Let $1/2 \leq a \leq b_4^{+}$. Then we have
\begin{align*}
&a(1-a)(1-2a)(1+3a-3a^2) \leq 0,\\
&\frac{1}{12}(1+9a-9a^2-90a^3+120a^4-30a^6) < 0,\\
&\frac{1}{36}(1+6a+6a^2-99a^3+180a^5-60a^6-30a^7) < 0.
\end{align*}
For $n \geq 4$, we have
\begin{align*}
&-5(1-a)^4a^4 < 0,\\
&\frac{n(n-1)}{12}(1+21a-111a^2+150a^3-60a^4)a^2 < 0,\\
&\frac{n(n-1)(n-2)}{36}(1+3a-39a^2+66a^3-30a^4)a < 0,\\
&-\frac{n(n-1)(n-2)(n-3)}{24}B_4(1-a) \leq 0.
\end{align*}
For the remaining part, it holds that
\begin{align*}
&-n(1-a)(1-10a+20a^2-10a^3)a^3\\
&\hspace{60pt} +\frac{n(n-1)}{12}(1+21a-111a^2+150a^3-60a^4)a^2\\
&\hspace{15pt} \leq n\Biggl\{-(1-a)(1-10a+20a^2-10a^3)a^3\\
&\hspace{60pt} +\frac{(4-1)}{12}(1+21a-111a^2+150a^3-60a^4)a^2 \Biggr\}\\
&\hspace{15pt} = \frac{n}{4}(1+17a-67a^2+30a^3+60a^4-40a^5)a^2 <0.
\end{align*}
Thus we obtain $e^{(a-1)x}g_3^{(5)}(a,x) <0$, that is, $\zeta(\sigma,a) <0$ for all $\sigma \in (-4,-3)$.\\

Finally if $b_4^{-} < a < 1/2$ or $b_4^{+} < a < 1$, $\zeta(\sigma,a)$ has zeros in $(-4,-3)$, and if $a=1$, $\zeta(\sigma, 1)$ has no zeros in $(-4,-3)$. This concludes the proof for the case of $N$ =3. \\

\noindent (2) $N \geq 4$ case.\\

We give a proof of the remaining cases of $N \geq 4$ by using Spira's theorem.

\begin{thm}$($\cite{Spi}$)$\label{Spi}
If $\sigma \leq -(4a+1+2[1-2a])$ and $|t| \leq 1$, then $\zeta(s,a) \neq 0$ except for zeros on the negative real axis, one in each interval $(-2n-4a-1, -2n-4a+1)$ for $n\geq 1-2a$.
\end{thm}

Spira showed this theorem by using the functional equation of the Hurwitz zeta function and Rouch\'{e}'s theorem in complex analysis. Shifting the intervals, we have the following statement.

\begin{cor}
Let $M \geq 2$ be an integer. The Hurwitz zeta function $\zeta(\sigma, a)$ has exactly one zero in each interval $[-2M-2, -2M)$.
\end{cor}

\begin{proof}
By the definition of the Hurwitz zeta function, we have $\zeta(s,1/2)$ = $(2^s-1)\zeta(s)$ and $\zeta(s,1)$ = $\zeta(s)$. Thus, for $a$ = 1/2 and 1, the point $s = -2M-2$ is the only one zero in each $[-2M-2, -2M)$. In other cases, the Hurwitz zeta function has an odd number of real zeros counting multiplicity in each $[-2M-2, -2M)$ by $B_{2M+1}(a)B_{2M+3}(a) <0$ (Lemma \ref{Bpm}) and Lemma \ref{BH}. However, by Theorem \ref{Spi}, the Hurwitz zeta function does not have more than two real zeros in each interval $[2M-2, -2M)$ whose length is 2.
\end{proof}

Theorem \ref{main1} for the cases of $N \geq 4$ follows readily from this corollary and Lemma \ref{BH}. Combining (1-i, ii, iii) and Theorem \ref{Naka}, we therefore obtain Theorem \ref{main1} for the all cases of $N \geq -1$.

\subsection*{Acknowledgement}
The author thanks Professor Masanobu Kaneko and Professor Takashi Nakamura for their helpful comments on an earlier version of this paper.


\begin{thebibliography}{HD}

\normalsize
\baselineskip=17pt


\bibitem{AIK} T. Arakawa, T. Ibukiyama, M. Kaneko, \emph{Bernoulli Numbers and Zeta Functions}, Springer Monographs in Mathematics, Springer, Tokyo, 2014, with an appendix by Don Zagier.
\bibitem{ES} K. Endo, Y. Suzuki, \emph{Real zeros of Hurwitz zeta-functions and their asymptotic behavior in the interval $(0,1)$}, arXiv:1705.08102.
\bibitem{Hur} A. Hurwitz,
\emph{Einige Eigenschaften der Dirichletschen Functionen $F(s) = \sum (\frac{D}{n}) \cdot \frac{1}{n^s}$ die bei der Bestimmung der Klassenzahlen bin{\"a}rer quadratischer Formen auftreten}, Z. Math. Phys., {27} (1882) 86--101.
\bibitem{Dav} D. J. Leeming,
\emph{The real zeros of the Bernoulli polynomials}, Journal of Approximation theory, {58} (1989) 124--150.
\bibitem{Naka1} T. Nakamura,
\emph{Real zeros of Hurwitz-Lerch zeta and Hurwitz-Lerch type of Euler-Zagier double zeta functions}, Math. Proc. Cambridge Philos. Soc., {160} (1) (2016) 39--50.
\bibitem{Naka2} T. Nakamura,
\emph{Real zeros of Hurwitz-Lerch zeta functions in the interval $(-1,0)$}, J. Math. Anal. Appl., {438} (1) (2016) 42--52.
\bibitem{Spi} R. Spira,
\emph{Zeros of Hurwitz zeta functions}, Mathematics of computation, {30} (136) (1976) 863--866.

\end{thebibliography}
\end{document}